\documentclass[11pt]{amsart}
\usepackage{amsmath}
\usepackage[all]{xy}
\usepackage{amssymb}
\usepackage{mathrsfs}
\usepackage{cite}
\usepackage{hyperref}
\usepackage{amsfonts}
\usepackage{mathdots}

\theoremstyle{plain}
\newtheorem{thm}{Theorem}[section]
\newtheorem{prop}[thm]{Proposition}
\newtheorem{lem}[thm]{Lemma}
\newtheorem{cor}[thm]{Corollary}

\theoremstyle{remark}
\newtheorem{rem}[thm]{Remark}

\theoremstyle{definition}
\newtheorem{defn}[thm]{Definition}

\theoremstyle{conjecture}

\def\rank{\operatorname{rank}}
\def\min{\operatorname{min}}
\def\Bs{\operatorname{Bs}}
\def\PBs{\operatorname{PBs}}
\def\Pic{\operatorname{Pic}}
\def\det{\operatorname{det}}
\def\NS{\operatorname{NS}}
\def\cL{\mathcal{L}}
\def\cO{\mathcal{O}}
\def\cR{\mathcal{R}}

\title{Indecomposability of the bounded derived categories of Brill-Noether varieties}
\author{Xun Lin, Chenglong Yu}
\email{lin-x18@mails.tsinghua.edu.cn}
\email{yuchenglong@mail.tsinghua.edu.cn}

\begin{document}

\begin{abstract}
We prove that the bounded derived category of coherent sheaves of the Brill-Noether variety $G^{r}_{d}(C)$ that parametrizing linear series of degree $d$ and dimension $r$ on a general smooth projective curve $C$ is indecomposable when $d\leq g(C)-1$.
\end{abstract}
\maketitle

\section{Introduction}
The semi-orthogonal decompositions of the bounded derived category of coherent sheaves have many applications to the geometry of the spaces. As there are many examples of semi-orthogonal decompositions \cite{sod1}\cite{kuznetsov2021semiorthogonal}, a naive question is that whether the bounded derived categories are indecomposable in sense of semi-orthogonal decompositions. The indecomposability is closely related to the geometric space of the Bridgeland stability conditions, see \cite[Section 5]{lin2021nonexistence}, and it is useful in many aspects. For example, if we know an embedding from a nice space to an indecomposable category, then the embedding must be an equivalence. In paper \cite{kawatani2018nonexistence}, the authors show that if the base locus of the canonical divisor is of zero dimension, then the bounded derived category is indecomposable. It is generalized to the singular varieties \cite{spence2021note}. The paper \cite{lin2021nonexistence} shows that if the para-canonical base locus of the canonical divisor is of zero dimension, then the derived category is indecomposable, in particular, it completes the picture of indecomposability of the derived categories of symmetric power of curves. Very recently, F. Caucci study the para-canonical base point of the canonical divisor, connected the indecomposability problems to the generic vanishing theorems, and obtained more examples whose bounded derived categories of coherent sheaves are indecomposable, including Hilbert scheme of points on certain surfaces \cite{caucci2021paracanonical}. Other techniques of studying the indecomposability of derived categories are to consider the families \cite{bastianelli2020indecomposability}, and new notions stably semi-orthogonal indecomposability \cite{pirozhkov2021stably}.

\par

Let $C$ be a smooth projective curve over $\mathbb{C}$. In this paper, we continue to study the indecomposability of derived categories of the Brill-Noether variety $G^{r}_{d}(C)$ which is the $d$-th symmetric power of $C$ when $r=0$. Recall that $G^{r}_{d}(C)$ is the scheme that parametrizes degree $d$ and dimension $r$ linear series on $C$. When $C$ is general, $G^{r}_{d}(C)$ is smooth projective, and of expected dimension. We prove the following.

\begin{thm}\label{mainA}
Let $C$ be a general curve. The bounded derived categories  $D^{b}(G^{r}_{d}(C))$ is indecomposable if $d\leq g(C)-1$.
\end{thm}

Very recently, Toda proved the Quot scheme formula \cite[Theorem 1.1]{toda2021derived} that was conjectured by Jiang \cite{jiang2021derived}, and obtained semi-orthogonal decompositions of derived categories of $G^{r}_{d}(C)$ as special cases.

\begin{prop}\label{quotscheme}\cite[Corollary 1.5]{toda2021derived}
Let $C$ be a general curve. There is a semi-orthogonal decomposition if $\delta\geq 0$,
$$D^{b}(G^{r}_{g-1+\delta}(C))=\Big\langle \tbinom{\delta}{i}\ \text{-copies of}\  D^{b}(G^{r-i}_{g-1-\delta}(C)),\ 0\leq i\leq \min(\{\delta,r+1\})\Big\rangle.$$
\end{prop}

The Theorem \ref{mainA} implies that the components in the decompositions of Proposition \ref{quotscheme} are indecomposable. We complete the picture of the indecomposability of bounded derived categories of $G^{r}_{d}(C)$ for general curve $C$.

\par

Our method follows from paper \cite{lin2021nonexistence}. Firstly we write down a first Chern class formula in Section 2. Since $G^{r}_{d}(C)$ is a special case of the canonical blow-up of the determinantal varieties, we start with a more general setting, namely, we compute a first Chern class formula for the canonical blow-up of the determinantal varieties in Section 2, see Proposition \ref{Chernd}.

\par

In Section 3, we analyze the deformation property of the relative hyperplane class which appears in the first Chern class formula. We relates the deformation of the relative hyperplane to an algebraically global generations of vector bundles that is a weaker notion of continuous globally generated \cite{PP}, see Proposition \ref{algdeform}.
Some other general theorems of indecomposability of derived categories are obtained, we hope that there will be more examples, see Theorem \ref{indecompoA} and Corollary \ref{indecompoB}.

\par

In Section 4, we study the examples $G^{r}_{d}(C)$, and prove the main Theorem \ref{mainA}. The key point is an algebraically global generation of dual of the push forward of certain Poincar\'e bundle. The method is to reduce the original problem of indecomposability to the concrete problems of linear series.

\section{The first Chern class formula}
In this section, we compute the first Chern class formula for the canonical blow-up of the determinantal varieties under some smoothness and dimension assumptions. The main examples are the classical Brill-Noether varieties.

\par

Let $X$ be a smooth projective variety, and $(E,F,\gamma)$ be a pair of vector bundles of rank $n$ and $m$ with bundle morphism $\gamma$. Define $X_{k}(\gamma)$ as the $k$-th degeneracy loci of $\gamma$, namely the subvariety of $X$ which parametrizes the the points $x$ such that the rank of $\gamma_{x}$ is less than or equal to $k$. Consider a closed subscheme of the relative Grassmanian $G(n-k, E)$ which parametrizes the pair $(x,W)$ that $x\in X_{k}(\gamma)$, $W\subset \operatorname{Ker}(\gamma_{x})$. We denote this variety as $X'_{k}(\gamma)$ which is usually called the canonical blow-up of $X_{k}(\gamma)$. We always assume that $X'_{k}(\gamma)$ is smooth and of expected dimension. For the precise definitions, see \cite[Chap II]{arbarello2013geometry}.

$$\xymatrix@C3cm{G(n-k,E)\ar[r]^{p}&X\\
X'_{k}(\gamma)\ar[ru]^{u}\ar[u]^{j}&}$$

Consider the Pl$\ddot{\text u}$cker embedding $G(n-k,E)\hookrightarrow \mathbb{P}(\wedge^{n-k}E)$. Let $H$ be the canonical relative hyperplane section of $\mathcal{O}(1)$ on
$\mathbb{P}(\wedge^{n-k}E)$.

\begin{prop}\label{Chernd}
The first Chern class formula of $X'_{k}(\gamma)$ is
$$-C_{1}(X'_{k}(\gamma))=(\rank{F}-\rank{E})\cdot H\vert_{X'_{k}(\gamma)} + u^{\ast}(-C_{1}(X) + (n-k)\cdot(C_{1}(F)-C_{1}(E))).$$
\end{prop}

\begin{proof}
We write $S$ as the tautological vector bundle on $G(n-k,E)$, and $Q$ as the universal quotient bundle. There is a natural exact sequence of
vector bundles
$$0\rightarrow S\rightarrow p^{\ast}E\rightarrow Q\rightarrow 0.$$
Applying $\otimes S^{\vee}$, we get an exact sequence of vector bundles
$$0\rightarrow S^{\vee}\otimes S\rightarrow S^{\vee}\otimes p^{\ast}E\rightarrow S^{\vee}\otimes Q\rightarrow 0.$$
We denote the last term $S^{\vee}\otimes Q$ as $M$, which is the vector bundle of tangent directions on the fibers of $p$.
There is a tangent bundle formula
$$0\rightarrow M\rightarrow T_{G(n-k,E)}\rightarrow p^{\ast}T_{X}\rightarrow 0.$$
We have $$C_{1}(M)=-C_{1}(S^{\vee}\otimes S)+ C_{1}(S^{\vee}\otimes p^{\ast}E)=C_{1}(S^{\vee}\otimes p^{\ast}E).$$
The second equality is because $C_{1}(S^{\vee}\otimes S)=\rank(S)\cdot C_{1}(S^{\vee})+\rank(S)\cdot C_{1}(S)=0$. \\
On the other hand,
$$C_{1}(T_{G(n-k,E)})=C_{1}(M)+p^{\ast}C_{1}(T_{X}).$$
Now we relate the Chern class of $G(n-k,E)$ with the Chern class of $X'_{k}(\gamma)$ by the normal bundle formula
$$0\rightarrow T_{X'_{k}(\gamma)}\rightarrow T_{G(n-k,E)}\vert_{X'_{k}(\gamma)}\rightarrow N_{X'_{k}(\gamma)\vert G(n-k,E)}\rightarrow 0.$$
Consider the morphisms of vector bundles
$$\xymatrix{S\ar[r]^{\hookrightarrow}&p^{\ast}E\ar[r]^{p^{\ast}\gamma}&p^{\ast}F}.$$
Denote the composition of morphisms of the vector bundles as $p^{\ast}\gamma$ again. By definition $X'_{k}(\gamma)$
is the zeros of the section $s\in \Gamma(S^{\vee}\otimes p^{\ast}F)$ that induced from $p^{\ast}\gamma$. Since we assume that $X'_{k}(\gamma)$ is of expected dimension and smooth, we have
$$N_{X'_{k}(\gamma)\vert G(n-k,E)}\cong S^{\vee}\otimes p^{\ast}F\vert_{X'_{k}(\gamma)}.$$
Therefore, we have Chern class formula
$$j^{\ast}C_{1}(G(n-k,E))=C_{1}(X'_{k}(\gamma))+j^{\ast}C_{1}(S^{\vee}\otimes p^{\ast}F).$$
Combining the formulas above, we get

\begin{align*}
    C_{1}(X'_{k}(\gamma))=& j^{\ast}C_{1}(G(n-k,E))-j^{\ast}C_{1}(S^{\vee}\otimes p^{\ast}F)\\
                         =& j^{\ast}(C_{1}(M)+p^{\ast}C_{1}(X)-C_{1}(S^{\vee}\otimes p^{\ast}F))\\
                         =& j^{\ast}(C_{1}(S^{\vee}\otimes p^{\ast}E)+p^{\ast}C_{1}(X)-C_{1}(S^{\vee}\otimes p^{\ast}F))\\
                         =& \rank(E)\cdot j^{\ast}C_{1}(S^{\vee})+\rank(S)\cdot u^{\ast}C_{1}(E)+ u^{\ast}C_{1}(X) - \rank(S)\cdot u^{\ast}C_{1}(F) \\
                         \ & - \rank(F)\cdot j^{\ast}C_{1}(S^{\vee}) \\
                         =& (\rank(E)-\rank(F))\cdot j^{\ast}C_{1}(S^{\vee})+ u^{\ast}C_{1}(X) + (n-k)\cdot u^{\ast}(C_{1}(E)-C_{1}(F)) \\
\end{align*}
Since $j^{\ast}C_{1}(S^{\vee})= H\vert_{X'_{k}(\gamma)}$,
therefore
$$-C_{1}(X'_{k}(\gamma))= (\rank(F)-\rank(E))\cdot H\vert_{X'_{k}(\gamma)}+ u^{\ast}(-C_{1}(X) + (n-k)\cdot(C_{1}(F)-C_{1}(E))).$$
\end{proof}

\section{The deformations of the relative hyperplane}
Let $E$ be a vector bundle of $\rank n$ on a smooth projective  variety $X$. Write $(\mathbb{P}(E),\mathcal{O}(1))$ as projective bundle of $E$ with the canonical universal bundle $\mathcal{O}(1)$.
$$p: \mathbb{P}(E)\longrightarrow X.$$

\par

 Motivated by the Chern class formula in Section 2, we study the algebraic class of $\mathcal{O}(1)$. Following paper \cite{lin2021nonexistence}, we write the notation $\PBs\vert L\vert$ for base point of the algebraic linear systems of $L$. Note here that the base point is independent of the choice of line bundle in an algebraic class. For convenient, we always define the base point of a class. Clearly, any algebraically trivial line bundle on $\mathbb{P}(E)$ restricts to be trivial on the fibers of $p$, hence must be pulled back of algebraically trivial line bundle on $X$. For linear series of $\mathcal{O}(1)$, we have a classical theorem.

\begin{prop}\label{ratdeform}
The following statements are equivalent,
\begin{enumerate}
    \item $\Bs\vert\mathcal{O}(1)\vert=\varnothing$.
    \item $E^{\vee}$ is globally generated.
\end{enumerate}
\end{prop}

\begin{proof}
There is an isomorphism by projection formula,
$$H^{0}(\mathbb{P}(E),\mathcal{O}(1))\cong H^{0}(X, E^{\vee}).$$

$(1)\implies (2):$ Given any closed point $x\in X$, since $\mathcal{O}(1)$ is base point free, there is a section $s_{1}$ such that $s_{1}\vert_{p^{-1}x}$ define a nontrivial hyperplane $H_{1,x}$ of $\mathbb{P}(E_{x})$. Choose a non-zero vector $v\in H_{1,x}$, again since $\mathcal{O}(1)$ is base point free, there must be a section $s_{2}$ defining a hyperplane $H_{2,x}$ of $\mathbb{P}(E_{x})$ such that $v\notin H_{2,x}$. The next step is to choose vector in $H_{2,x}\cap H_{1,x}$, and then get $H_{3,x}$. By induction, we get sections $\langle s_{1}, \cdots , s_{n} \rangle$. They are linearly independent in $E^{\vee}_{x}$. Suppose not, then assume not all zero $\alpha_{i}$ define
$$\sum \alpha_{i}\cdot s_{i,x}=0.$$
By construction, take $v\in \cap_{\leq n-1} H_{i,x}$, then $s_{i,x}(v)=0$ for $i\leq n-1$, but $s_{n,x}(v)\ne 0$, hence $\alpha_{n}=0$. Similarly, $$\alpha_{1}=\alpha_{2}=\cdots =\alpha_{n-1}=0,$$
a contradiction. Thus $\langle s_{1}, \cdots , s_{n} \rangle$ generates $E^{\vee}_{x}$.
\par
$(2)\implies (1):$ Given any point $(x,v)$ such that $v\in \mathbb{P}(E_{x})$, since $E^{\vee}$ is globally generated, there are sections $\langle s_{1}, \cdots , s_{n}\rangle$ of $E^{\vee}$ that generate $E^{\vee}_{x}$. Therefore there is a section $s_{j}$ such that $s_{j,x}(v)\ne 0$. This implies $\mathcal{O}(1)$ is base point free.
\end{proof}

We get a similar theorem for $\PBs\vert\mathcal{O}(1)\vert$. Before the statement, we firstly define algebraically globally generated. It is a weaker notion than the continuous globally generated  \cite{PP}.
\begin{defn}
The vector bundle $E$ of $\rank n$ is algebraically globally generated if
for any closed point $x$ of $X$, there exists $n$ algebraically trivial line bundles $\{L_{1}, \cdots, L_{n}\}$, and sections $s_{i}$ of $E\otimes L_{i}$ such that the local sections $\frac{s_{1}}{s'_{i}}, \cdots, \frac{s_{n}}{s'_{n}}$ generate the vector space $E_{x}$. Here $s'_{i}$ is non-vanishing local section of $L_{i}$ under some trivialization around $x$. Note that the generation property is independent of the choices of local trivializations.
\end{defn}

\begin{rem}
If $E$ is a line bundle, then $E$ is algebraically globally generated if and only if $\PBs\vert E\vert=\varnothing$.
\end{rem}

\begin{prop}\label{algdeform}
The following statements are equivalent,
\begin{enumerate}
    \item $\PBs\vert\mathcal{O}(1)\vert=\cap_{L\in \Pic^{0}(X)}\Bs\vert\mathcal{O}(1)\otimes p^{\ast}L\vert=\varnothing$.
    \item $E^{\vee}$ is algebraically globally generated.
\end{enumerate}
\end{prop}

\begin{proof}
Given a line bundle $L$ on $X$, there is a projection formula,
$$H^{0}(\mathbb{P}(E),\mathcal{O}(1)\otimes p^{\ast}L)\cong H^{0}(X, E^{\vee}\otimes L).$$

The proof is essentially the same with the proof of Proposition \ref{ratdeform}.
\end{proof}

\begin{cor}\label{wedgeglobal}
If a vector bundle $E$ is algebraically globally generated, then $\wedge^{r}E$ is also algebraically globally generated. In particular, $\PBs\vert\det E\vert=\varnothing$.
\end{cor}

\begin{proof}
Let $x$ be any closed point of $X$. Since $E$ is algebraically globally generated, there exists sections $s_{i}$ of $E\otimes L_{i}$ where $L_{i}$ are algebraically trivial line bundles such that $\langle\frac{s_{i}}{s'_{i}}\rangle$ generates $E_{x}$. Here $s'_{i}$ are the local non-vanishing sections of $L_{i}$. Construct the wedge product of sections

$$\{s_{i_{1}}\wedge s_{i_{2}}\wedge \cdots \wedge s_{i_{r}}\}_{i_{1}\textless i_{2}\textless \cdots \textless i_{r} }$$
which are sections of $\wedge^{r}E\otimes^{n}_{1} L_{i}$. After quotient the local non-vanishing section of $\otimes^{n}_{1}L_{i}$, they generate $\wedge^{r}E_{x}$.
\end{proof}

Turning back to $X'_{k}(\gamma)$, and we use the notations in Section 2. Observe that the property that $E^{\vee}$ is algebraically globally generated is stronger than that $\PBs\vert H\vert_{X'_{k}(\gamma)}\vert=\varnothing$, namely if $E^{\vee}$ is algebraically globally generated, then by Corollary \ref{wedgeglobal}, $\wedge^{r}E^{\vee}$ is also algebraically generated, and hence $\PBs\vert\mathcal{O}(1)\vert=\varnothing$, where $\mathcal{O}(1)$ is the universal bundle on $\mathbb{P}(\wedge^{r}E)$.

\par

Let $X'_{k}(\gamma)$ satisfies the assumption in Section 2. We translate the condition
$\PBs\vert H\vert_{X'_{k}(\gamma)}\vert=\varnothing$.

\begin{prop}
$\PBs\vert H\vert_{X'_{k}(\gamma)}\vert=\varnothing$ if the following is true: Given any closed point $y\in X'_{k}(\gamma)$, namely a pair $(x,W)$ such that
$x\in X'_{k}$ and $W\subset \ker \gamma_{x}$, $\dim W=n-k$, there is a section $s_{L}$ of $\wedge^{n-k}E^{\vee}\otimes L$ for some algebraically trivial line bundle $L$ such that
$\frac{s_{L}}{s'_{L}}(\det W)\ne 0$. Here $s'_{L}$ is the local non-vanishing section of $L$ around $x$.
\end{prop}

\begin{proof}
The condition is saying that given any closed point $(x,W)$, there always exists an effective divisor $H_{L}$ which are algebraically equivalent to $\mathcal{O}(1)$ such that $(x,W)\notin H_{L}$. The section $H_{L}\vert_{X'_{k}(\gamma)}$ does not vanish at $(x,W)$, therefore $\PBs\vert H\vert_{X'_{k}(\gamma)}\vert=\varnothing$.
\end{proof}

\begin{thm}\label{indecompoA}
Assume that $\rank F\geq \rank E$, $E^{\vee}$ is algebraically globally generated, and
$$\PBs \vert -C_{1}(X)+(n-k)\cdot (C_{1}(F)-C_{1}(E))\vert=\varnothing.$$
Then $D^{b}(X'_{k}(\gamma))$ is indecomposable.
\end{thm}

\begin{proof}
According to Proposition \ref{Chernd} and \cite[Corollary 1.6]{lin2021nonexistence}, we only need that $\PBs\vert H_{X'_{k}(\gamma)}\vert=\varnothing$, where $H$ is the relative Hyperplane of $\mathcal{O}(1)$ on $\mathbb{P}(\wedge^{n-k}E)$. If $E^{\vee}$ is algebracically generated, then by Corollary \ref{wedgeglobal}, $\wedge^{n-k}E$ is also algebraically globally generated. Therefore, by Theorem \ref{algdeform}, we have $\PBs\vert H\vert=\varnothing$, and then $\PBs\vert H\vert_{X'_{k}(\gamma)}\vert=\varnothing$.
\end{proof}

\begin{lem}
Let $X$ be an abelian variety. If a class $[L]$ of $\NS(X)$ is represented by an effective divisor, then $\PBs\vert[L]\vert=\varnothing$.
\end{lem}

\begin{proof}
Assume that an effective divisor $D$ represents $[L]$. Write $D=\sum a_{i}\cdot P_{i}$ where $P_{i}$ are the prime divisors, and $a_{i}$ are non-negative integers. According to
\cite[Corollary 4.4]{lin2021nonexistence}, $\PBs\vert P_{i}\vert=\varnothing$, therefore $\PBs\vert D\vert=\varnothing$. Since the para-canonical base point is independent of choice of line bundle in a class, we have $\PBs\vert [L]\vert=\varnothing$.
\end{proof}

\begin{cor}\label{indecompoB}
Let $X$ be an abelian variety, and $(E,F,\gamma)$ be vector bundle on $X$ with morphism $\gamma$. Assume $X'_{k}(\gamma)$ is smooth and of expected dimension, $\rank F\geq \rank E$, $C_{1}(F)-C_{1}(E)$ can be represented by an effective divisor, and $E^{\vee}$ is algebraically globally generated, then $D^{b}(X'_{k}(\gamma))$ is indecomposable.
\end{cor}

	\section{The Brill-Noether variety}
	In this section, we study the scheme $G_d^r(C)$ parametrizing degree $d$ dimension $r$ linear series on a general algebraic curve $C$ of genus $g$. We prove that $D^b(G_d^r(C))$ is indecomposible for general curve $C$ if $d\leq g-1$.
	
	First we recall the construction of $G_d^r(C)$. Let $X=\Pic^{d}(C)$ be the Picard variety parametrizing degree $d$ line bundles on $C$. There exists a Poincar\'e line bundle $\cL$ on $C\times X$ such that for each element $L\in X$, the restriction $\cL|_{C\times \{L\}}$ is exactly $L$ under the isomorphism $C\times{L}\cong C$. Fix an effective divisor $D$ on $C$ with degree $2g-d-1$. Let $\Gamma=D\times X$ be an effective divisor on $C\times X$. Denote by $\nu$ the projection $C\times X\to X$. Then $E=\nu_* \cL(\Gamma)$ and $F=\nu_*( \cL(\Gamma)/\cL)$ are vector bundles on $X$ with $\rank{E}=g$ and $\rank{F}=2g-d-1$. There is an natural morphism $\gamma\colon E\to F$ induced by $\cL(\Gamma)\to \cL(\Gamma)/\cL$. Then $G_d^r(C)=X^\prime_{g-r-1}(\gamma)$. Griffiths and Harris \cite{griffiths1980variety} showed that each component of $G_d^r(C)$ has expected dimension $\rho= g-(r+1)(g-d+r)$ for generic curve $C$. Gieseker \cite{gieseker1982stable} proved that $G_d^r(C)$ is smooth for generic curve $C$. Noted that the Poincar\'e bundle $\cL$ is uniquely determined up twist by pull back of line bundles $\cR$ on $X$. The twist of $\nu^*\cR$ does change $E$ but does not change the variety $G_d^r(C)$. For any point $p\in C$, let $i_p$ be the section $X\to C\times X$ defined by $L\mapsto (p, L)$. Denote by $L_p=i_p^*(\cL(\Gamma))$. Here we may assume that for some point $p_0\in C$, the bundle $L_{p_0}\cong \cO_{X}$. Otherwise we replace $\cL$ by $\cL\otimes \nu^*(L_{p_0})^{-1}$. In order to apply our theorem \ref{indecompoA} to $G_d^r(C)$, we show the following.
	
	\begin{prop}
		\label{prop: global generation of Poincare bundle}
		Under the assumption $L_{p_0}\cong \cO_{X}$, the vector bundle $E^\vee$ constructed above is algebraically globally generated.
	\end{prop}
	
	\begin{proof}
		The Poincar\'e line bundle $\cL(\Gamma)$ defines a morphism $C\to \Pic(X)$ by sending $p$ to $L_p$. Since $L_{p_0}\cong \cO_X$, we have $L_p\in \Pic^0(X)$ for all $p\in C$. For any point $p\in C$, we construct a section $s_p$ of $E^\vee\otimes L_p$ as follows. Since $i_p^*(\cL(\Gamma)\otimes \nu^*(L_p^{-1}))\cong \cO_X$, we can choose a trivializing section $t_p$ of $\cL(\Gamma)\otimes \nu^*(L_p^{-1})|$ on ${\{p\}\times X}$. Construct the section $s_p\colon E\otimes L_p^{-1} \to  \mathbb{C}$ on $X$ pointwisely by
		\begin{eqnarray*}
			s_p|_L \colon H^0(C, L(D))&\to& \mathbb{C}    \\
			v &\mapsto& {v|_p\over t_p|_{(p,L)}} .
		\end{eqnarray*}
		Let $a_i$ a nonzero section of $L_p$ at point $L$. We claim that for generic $(p_1,\cdots,  p_g)\in C_g$, the sections $(s_{p_1}|_L)\otimes a_1, \cdots , (s_{p_g}|_L)\otimes a_g$ form a basis of $E^\vee|_L$. The claim follows from the following lemma:
		\begin{lem}
			\label{lemma: linear indepdence of sections on points}
			Assume $v_0, v_1,\cdots, v_r$ are linearly independent sections of a line bundle $L$ over an algebraic variety $C$. Fixing a nonzero section $b_p$ of $L$ at each point $p\in C$. Denote by $v_i(p)={v_i\over b_p}\in k$. Then for general points $(p_0, p_1, \cdots, p_r)\in C^{(r+1)}$, the matrix $\{v_i(p_j)\}_{(r+1)\times (r+1)}$ is invertible. (The invertibility does not depend on the choice of $b_p$.)
		\end{lem}
		
		\begin{proof}
			We prove the lemma by induction on $r$. When $r=0$, $v_0(p)\neq 0$ for a generic point $p\in C$. Assume the lemma is true for $r\geq 0$. Consider $r+2$ linearly independent sections $v_0, v_1,\cdots, v_r, v_{r+1}$. For general points $(p_0, p_1, \cdots, p_r)\in C^{(r+1)}$, denote by $B=\{v_i(p_j)\}_{0\leq i \leq r, 0\leq j\leq r}$ an invertible matrix . Let
			$$v_{r+1}^\prime= v_{r+1}-(v_{r+1}(p_0), \cdots, v_{r+1}(p_r))B^{-1}(v_0,\cdots, v_r)^T.$$
			Then $v_{r+1}^\prime \neq 0$ because $v_0, v_1,\cdots, v_r, v_{r+1}$ are linearly independent, hence $v_{r+1}^\prime(p_{r+1})\neq 0$ for generic $p_{r+1}\in C$. On the other hand, we have $v_{r+1}^\prime(p_i)=0$ for any $0\leq i\leq r$. Replacing $v_{r+1}$ by $v_{r+1}^\prime$ does not change the invertibility of $\{v_i(p_j)\}_{0\leq i \leq r+1, 0\leq j\leq r+1}$. So the statement is proved for $r+1$.
		\end{proof}
		
		Apply the lemma to $r=\dim H^0(C, L)-1$, we obtain the claim and hence the proposition is proved.
		
	\end{proof}
	
	\begin{rem}
		When $r=0$, the variety $G_d^0(C)=C_d$ and from the construction of section $s_p$, we can conclude that the zero locus of $s_p$ is the image of $C_{d-1}\to C_d$ via map $D\mapsto D+p$. Combining with the Chern class formula we obtain that $-C_1(C_d)=[\Theta]+(g-d-1)[C_{d-1}]$. This recovers results in \cite[Lemma 2.1]{BGL}.
	\end{rem}
	
	\begin{rem}
	    The bundle $E^\vee$ is not globally generated. Otherwise, $\det E^\vee\cong \cO_X(\Theta)$ is also globally generated, where $\Theta$ is theta divisor on $X$. This contradicts with the fact that $\dim H^0(X, \cO_X(\Theta))=1$ and $ \cO_X(\Theta)$ has $\Theta$ as base locus.
	\end{rem}
	
	\begin{thm}
		\label{theorem: Brill-Noether indecomposible}
		For a generic curve $C$, the derived category $D^b(G_d^r(C))$ is indecomposable for $d\leq g-1$.
	\end{thm}
	
	\begin{proof}
		The first Chern class of $E$ is $C_1(E)=-[\Theta]$, where $\Theta$ is the theta divisor on $X$ and $E$ has rank $g$. The vector bundle $F$ is trivial a vector bundle of rank $2g-d-1$. So $\rank{F}\geq \rank{E}$ for $d\leq g-1$. And $-C_1(X)+(r+1)(C_1(F)-C_1(E))=(r+1)[\Theta]$ is an effective class. Proposition \ref{prop: global generation of Poincare bundle} implies that $E^\vee$ is algebraically globally generated. So $D^b(G_d^r(C))$ is indecomposable for $d\leq g-1$ according to Corollary \ref{indecompoB}.
		
	\end{proof}

\end{document}